\documentclass{amsart}
\usepackage{amssymb}
\usepackage{amsmath}
\usepackage{amsfonts, latexsym}
\usepackage{color,graphics}
\usepackage{graphicx}

\newtheorem{theorem}{Theorem}[section]

\newtheorem{lemma}[theorem]{Lemma}

\newtheorem{definition}[theorem]{Definition}
\newtheorem{example}[theorem]{Example}

\begin{document}
\author{Hee Sun Jung$^1$ and Ryozi Sakai$^2$}
\address{$^{1}$Department of Mathematics Education, Sungkyunkwan University,
Seoul 110-745, Republic of Korea.}
\email{hsun90@skku.edu}
\address{$^{2}$Department of Mathematics, Meijo University, Nagoya 468-8502, Japan.}
\email{ryozi@crest.ocn.ne.jp}

\title[]
{Pointwise Convergence of Fourier-type Series with Exponential Weights}
\date{\today}
\maketitle

\begin{abstract}
 Let $\mathbb{R}=(-\infty,\infty)$, and let $Q\in C^1(\mathbb{R}): \mathbb{R}\rightarrow[0,\infty)$
 be an even function. We consider the  exponential weights $w(x)=e^{-Q(x)}$,  $x\in \mathbb{R}$.
 In this paper we obtain a pointwise convergence theorem for the Fourier-type series with respect to
 the orthonormal polynomials $\left\{p_n(w^2;x)\right\}$.
\end{abstract}

Keywords; exponential weights, partial sum of Fourier-type series

\setcounter{equation}{0}
\section{Introduction and Theorem}

Let $\mathbb{R}=(-\infty,\infty)$, and
let $Q\in C^1(\mathbb{R}): \mathbb{R}\rightarrow[0,\infty)$
be an even function. We consider the weights $w(x):=\exp(-Q(x))$.
Then we suppose that $\int_0^\infty x^nw^2(x)dx<\infty$ for all $n=0,1,2,\ldots$.

First we need the following definition from \cite{[4]}. We say that $f: \mathbb{R}\rightarrow \mathbb{R^+}$
is quasi-increasing if there exists $C>0$ such that $f(x)\le Cf(y),  0<x<y$.
\begin{definition}[see \cite{[4]}]\label{Definition1.1}
Let $Q:  {\mathbb{R}}\rightarrow {\mathbb{R}}^+$ be a continuous even function satisfying the following properties:
\item[\,\,\,(a)] $Q'(x)$ is continuous in ${\mathbb{R}}$ and $Q(0)=0$.
\item[\,\,\,(b)] $Q''(x)$ exists and is positive in ${\mathbb{R}}\setminus\left\{0\right\}$.
\item[\,\,\,(c)] $  \lim_{x\rightarrow \infty}Q(x)=\infty.$
\item[\,\,\,(d)] The function
\begin{equation*}
T(x):=\frac{xQ'(x)}{Q(x)}, \quad x\neq 0
\end{equation*}
is quasi-increasing in $(0,\infty)$ with
\[
T(x)\ge \Lambda>1, \quad x\in {\mathbb{R}}^+\setminus\left\{0\right\}.
\]
\item[\,\,\,(e)] There exists $C_1>0$ such that
\begin{equation*}
\frac{Q''(x)}{|Q'(x)|}\le C_1\frac{|Q'(x)|}{Q(x)},  \quad a.e. \quad x\in {\mathbb{R}}\setminus\left\{0\right\}.
\end{equation*}
Then we say that $w=\exp(-Q)$ is in the class $\mathcal{F}(C^2)$.
Besides, if there exists a compact subinterval $J(\ni 0)$ of ${\mathbb{R}}$ and $C_2>0$ such that
\begin{equation*}
\frac{Q''(x)}{|Q'(x)|}\ge C_2\frac{|Q'(x)|}{Q(x)},  \quad a.e. \quad  x\in {\mathbb{R}}\setminus J,
\end{equation*}
then we say that $w=\exp(-Q)$ is in the class $\mathcal{F}(C^2+)$.
If $T(x)$ is bounded,
then $w$ is called the Freud-type weight, and if $T(x)$ is unbounded, then $w$ is the Erd\"os-type weight.
\end{definition}

A typical example in $\mathcal{F}(C^2+)$ is given as follows:
\begin{example}[{\cite[Example 1.2]{[4]} and \cite[Theorem 3.1]{[1]}}]\label{Example1.2}
{\rm
For $\alpha>1$ and a non-negative integer  $\ell$, we put
\begin{equation*}
Q(x)=Q_{\ell,\alpha}(x):=\exp_{\ell}(|x|^{\alpha})-\exp_{\ell}(0),
\end{equation*}
where for $\ell \ge 1$,
\begin{equation*}
\exp_{\ell}(x):=\exp(\exp(\exp(\cdots\exp x)\ldots))  \quad (\ell \textrm{-times})
\end{equation*}
and $\exp_0(x):=x$.
} \end{example}

   We construct the orthonormal polynomials $p_n(x)=p_n(w^2,x)$ of degree $n$ for $w^2(x)$, that is,
\begin{equation*}
   \int_{-\infty}^\infty p_n(x)p_m(x)w^2(x)dx=\delta_{mn} \quad (\textrm{Kronecker delta}).
\end{equation*}
Let $fw\in L_1(\mathbb{R})$. The Fourier series of $f$ is defined by
\begin{equation*}
  \tilde{f}(x):=\sum_{k=0}^\infty a_kfp_k(x),  a_kf:=\int_{-\infty}^\infty f(t)p_k(t)w^2(t)dt.
\end{equation*}
We denote the partial sum of $\tilde{f}(x)$ by
\begin{equation*}
  s_n(f,x):=s_n(w^2,f,x):=\sum_{k=0}^{n-1} a_kfp_k(x).
\end{equation*}
The partial sum $s_n(f,x)$ admits the representation
\begin{equation*}
  s_n(f,x)=\int_{-\infty}^\infty f(t)K_n(x,t)w^2(t)dt,
\end{equation*}
where
\begin{equation*}
  K_n(x,t)=\sum_{k=0}^{n-1} p_k(x)p_k(t).
\end{equation*}
Since
\begin{equation*}
  \int_{-\infty}^\infty K_n(x,t)w^2(t)dt=1,
\end{equation*}
we have
\begin{equation}\label{eq1.1}
  s_n(f,x)-f(x)=\int_{-\infty}^\infty K_n(x,t)(f(t)-f(x))w^2(t)dt.
\end{equation}
The Christoffel-Darboux formula asserts that
\begin{equation}\label{eq1.2}
K_n(x,t)=\frac{\gamma_{n-1}}{\gamma_n}\frac{p_n(x)p_{n-1}(t)-p_{n-1}(x)p_n(t)}{x-t},
\quad p_n(x)=:\gamma_n x^n+... .
\end{equation}
   In this paper we will show a pointwise convergence for the partial sum $s_n(f,x)$ of $\tilde{f}(x)$.
Let $g:\mathbb{R}\rightarrow\mathbb{R}$ be a function having bounded variation
on every compact interval.
The measure introduced by $g$ on Borel subset of $\mathbb{R}$ will be denoted by $|dg|$.
For any interval (finite or infinite) $I$, we define
\begin{equation*}
  V_\delta(I,g):=\int_I w^\delta(t)|dg(t)|,
\end{equation*}
where $0<\delta\le 1$ is fixed. Let $\mathcal{B}_\delta$ denote the class of all functions $g$ having bounded variation on $\mathbb{R}$, that is, $V_\delta(\mathbb{R},g)<\infty$.
We need the Mhaskar-Rakhmanov-Saff  numbers $a_x$;
\begin{equation*}
       x=\frac{2}{\pi}\int_0^1\frac{a_xuQ'(a_xu)}{(1-u^2)^{1/2}}du, \quad x>0.
\end{equation*}

   Mhaskar \cite{[5]} got the following pointwise convergence theorem.

\medskip
\noindent
{\bf Mhaskar Theorem (\cite[Theorem 9.1.2]{[5]})}.
Let $w=\exp(-Q)$ be a Freud-type weight
such that $Q''$ is increasing on $(0,\infty), f\in\mathcal{B}_1$, and let $x$ be a point of continuity of $f$.
Then for $n\ge cxQ'(x)$,
\begin{eqnarray*}
&&|s_n(f,x)-f(x)|\\
&\le& C\exp\left(cxQ'(x)\right)
\left\{\frac{1}{n}\sum_{k=1}^n V_1\left(\left[x-\frac{a_n}{k},x+\frac{a_n}{k}\right],f\right)
+\int_{|u|\ge c_1a_n}w(u)|df(u)|\right\},
\end{eqnarray*}
where $c$, $c_1$ and $C$ are some constants.
In particular, the sequence $\left\{s_n(f,x)\right\}$ converges to $f(x)$.

\medskip

We consider the Mhaskar Theorem
for the Erd\"os-type weight $w=\exp(-Q)\in \mathcal{F}(C^2+)$.
\begin{theorem}\label{Theorem1.3}
Let $w=\exp(-Q)\in \mathcal{F}(C^2+)$, and let $T(x)$ be unbounded.
We suppose $f\in \mathcal{B}_\delta$, $0<\delta<1$.
When $x$ is a point of continuity of $f$, there exist $C>0$, $c>0$ and $0 < d \le 1$ such that
\begin{eqnarray} \label{eq1.3} \nonumber
&&|s_n(f,x)|\\ \nonumber
&\le&  C\exp\left(cxQ'(x)\right)\\
&&\times \bigg(
\frac{1}{n}\sum_{k=1}^n V_\delta\left(\left[x-\frac{a_n}{k},x+\frac{a_n}{k}\right],f\right)
+\frac{1}{n}\int_{|u|\le a_{dn}}w^{\delta}(u)\left|df(u)\right| \\  \nonumber
&&\quad +\frac{1}{nT^{1/4}(a_n)}\int_{|u| \le a_{\frac{dn}{2}}}w(u)\left|df(u)\right|
+\frac{1}{T^{1/4}(a_n)}\int_{|u|\ge a_{\frac{dn}{2}}}w(u)\left|df(u)\right|\bigg).
\end{eqnarray}
Hence we have
\begin{eqnarray}\label{eq1.4} \nonumber
&&|s_n(f,x)-f(x)|\\\nonumber
&\le&  C\exp\left(cxQ'(x)\right) \\
&&\times \bigg(
\sqrt{\frac{a_n}{n}} V_\delta\left(\left[x-a_n,x+a_n\right],f\right)
+V_\delta\left(\left[x-\sqrt{\frac{a_n}{n}},x+\sqrt{\frac{a_n}{n}}\right],f\right)\\ \nonumber
&& \qquad +\frac{1}{n}\int_{|u|\le a_{dn}}w^{\delta}(u)\left|df(u)\right|
 +\frac{1}{nT^{1/4}(a_n)}\int_{|u| \le a_{\frac{dn}{2}}}w(u)\left|df(u)\right| \\ \nonumber
&&\qquad
+\frac{1}{T^{1/4}(a_n)}\int_{|u|\ge a_{\frac{dn}{2}}}w(u)\left|df(u)\right|\bigg).
\end{eqnarray}
  In particular, the sequence $\left\{s_n(f,x)\right\}$ converges to $f(x)$.
\end{theorem}
For any nonzero real valued functions $f(x)$ and $g(x)$, we write $f(x)\sim g(x)$ if there exist the constants $C_1, C_2>0$ independent of $x$  such that $C_1 g(x)\le f(x)\le C_2g(x)$ for all $x$. Similarly, for any two sequences of positive numbers $\left\{c_n\right\}_{n=1}^\infty$ and $\left\{d_n\right\}_{=1}^\infty$ we define $c_n\sim d_n$.\\
Throughout this paper $C, C_1, C_2,...$ denote positive constants independent of $n, x, t$ or polynomials $P_n(x)$. The same symbol does not necessarily denote the same constant in different occurences.

\setcounter{equation}{0}
\section{Lemmas}
   To prove Theorem \ref{Theorem1.3} we need some lemmas.
In this paper we treat $w=\exp(-Q)\in\mathcal{F}(C^2+)$. To prove our main theorem, we use many lemmas.
\begin{lemma}\label{Lemma2.1}
{\rm (1)} \cite[Lemma 3.5 (3.27)-(3.29)]{[4]} For fixed $L>0$ and uniformly for $t>0$,
\begin{equation*}
  a_{Lt}\sim a_t \quad \textrm{ and } \quad T(a_{Lt})\sim T(a_t).
\end{equation*}
\item[\,\,\,(2)]\cite[Lemma 3.4 (3.18),(3.17), Lemma 3.8 (3.42)]{[4]} For $t>0$,
\begin{equation*}
Q(a_t)\sim \frac{t}{\sqrt{T(a_t)}}
\quad \textrm{ and } \quad  Q'(a_t)\sim \frac{t\sqrt{T(a_t)}}{a_t}.
\end{equation*}
\item[\,\,\,(3)]\cite[Lemma 3.4 (3.4)]{[4]}  There exist $C_1, C_2$ such that for $s/r\ge 1$,
\begin{equation*}
  \left(\frac{s}{r}\right)^{\max\left\{\Lambda,C_1T(r)\right\}}
  \le \frac{Q(s)}{Q(r)}\le \left(\frac{s}{r}\right)^{C_2T(r)}.
\end{equation*}
\item[\,\,\,(4)]\cite[Lemma 3.11 (a),(b)]{[4]} Given fixed $0<\alpha$, we have uniformly for $t>0$,
\begin{equation*}
  \left|1-\frac{a_{\alpha t}}{a_t}\right|\sim \frac{1}{T(a_t)},
\end{equation*}
and there exists $C>0$ such that for $t>0$,
\begin{equation*}
  \left|1-\frac{a_t}{a_{st}}\right|
\ge \frac{C}{T(a_t)}\left|1-\frac{1}{s}\right|,  \quad \frac{1}{2}\le s\le 2.
\end{equation*}
In addition, for $0 < \alpha <1$, there exists $C >0$ such that for $s >0$,
\begin{equation*}
 T(x)\left(1-\frac{x}{a_s}\right) \ge C, \quad x \in [0, a_{\alpha s}].
\end{equation*}
\item[\,\,\,(5)]
\cite[Lemma 3.7]{[4]}
For some $\varepsilon >0$, and for large enough $t$,
\begin{equation}\label{eq2.1}
T(a_t) \le C t^{2-\varepsilon}.
\end{equation}
\item[\,\,\,(6)]
\cite[Theorem 3.5 (C)]{[4]}
For $t \ge r >0$ we have
\begin{equation*}
\frac{a_t}{a_r} \le C\left(\frac{t}{r}\right)^{1/\Lambda}.
\end{equation*}
\end{lemma}
We define
\begin{equation*}
  \varphi_u(x)=
\begin{cases}
\frac{a_u}{u}\frac{1-\frac{|x|}{a_{2u}}}{\sqrt{1-\frac{|x|}{a_u}+\delta_u}}, & |x|\le a_u; \\
\varphi_u(a_u),& a_u<|x|,
\end{cases}
\end{equation*}
where
\begin{equation*}
  \delta_u=\left\{uT(a_u)\right\}^{-2/3}.
\end{equation*}
Let $0<p<\infty$. The $L_p$ Christoffel functions $\lambda_{n,p}(w;x)$ with a weight $w$ are defined as follows;
\begin{equation*}
  \lambda_{n,p}(w;x):=\inf_{P\in \mathcal{P}_{n-1}}\int_{-\infty}^\infty|Pw|^p(u)du/|P|^p(x).
\end{equation*}
Then we have
\begin{equation*}
  \lambda_{n,2}(w;x)=\frac{1}{K(x,x)}=\frac{1}{\sum_{j=0}^{n-1}p_k(w^2,x)}
\end{equation*}
(see \cite[(9.14),(9.15)]{[4]}).
We denote the zeros of the orthonormal polynomial $p_n(w^2,x)$ by $x_{n,n}<x_{n-1,n}<...<x_{1,n}$.
Then we define the Christoffel numbers $\lambda_{k,n}, k=1,2,...,n$
such as $\lambda_{k,n}:=\lambda_{n,2}(w,x_{k,n})$.

\begin{lemma}[{\cite[Theorem 9.3 (c)]{[4]}}]\label{Lemma2.2}
  Let $0<p<\infty$. Let $L>0$. Then uniformly for $n\ge 1$ and $|x|\le a_n(1+L\delta_n)$, we have
\begin{equation*}
  \lambda_{n,p}(w;x)\sim \varphi_n(x)w^p(x).
\end{equation*}
\end{lemma}
\begin{lemma}\label{Lemma2.3}
{\rm (a)}\cite[ Corollary 13.4, (12.20)]{[4]}
Uniformly for $n\ge 1, 1\le k\le n-1$,
\begin{equation*}
  x_{kn}-x_{k+1,n}\sim \varphi_n(x_{k,n})   \quad \textrm{ and } \quad
1-\frac{x_{1n}}{a_{n}}\sim \delta_{n}.
\end{equation*}
Moreover,
\begin{equation*}
  \varphi_n(x_{k,n})\sim \varphi_n(x_{k+1,n}),  k=1,2,...,n-1.
\end{equation*}
\item[\,\,\,(b)] \cite[Lemma 3.4 (d)]{[2]}
Let $\max\left\{|x_{k,n}|,|x_{k+1,n}|\right\}\le a_{n/2}$. Then we have
for $x_{k+1,n}\le x \le x_{kn}$
\begin{equation*}
  w(x_{k,n})\sim w(x_{k+1,n})\sim w(x).
\end{equation*}
So, for given $C>0$ and $|x|\le a_{n/3}$, if $|x-x_{k,n}|\le C\varphi_n(x)$, then we have
\begin{equation*}
  w(x)\sim w(x_{k,n}).
\end{equation*}
\end{lemma}
\begin{lemma}[{\cite[Lemma 3.4]{[6]}}] \label{Lemma2.4}
For a certain constant $C>0$,
\begin{equation*}
  \frac{a_n}{n}\frac{1}{\sqrt{T(x)}}\varphi_n^{-1}(x)\le C.
\end{equation*}
\end{lemma}
\begin{lemma}[{\cite [Theorem 1.17, Theorem 1.18]{[4]}}]\label{Lemma2.5}
Uniformly for $n\ge 1$ we have
\begin{equation*}
  \sup_{x\in \mathbb{R}}|p_n(x)w(x)|x^2-a_n^2|^{1/4}\sim 1,
\end{equation*}
and for $w(x)=\exp(-Q(x))\in \mathcal{F}(C^2+)$,
\begin{equation*}
  \sup_{x\in \mathbb{R}}\left|p_n(x)w(x)\right|\sim a_n^{-1/2}(nT(a_n))^{1/6}.
\end{equation*}
\end{lemma}
\begin{lemma}[{\cite[Lemma 13.9]{[4]}}]  \label{Lemma2.6}
Uniformly for $n\ge 1$,
\begin{equation*}
  \frac{\gamma_{n-1}}{\gamma_n}\sim a_n.
\end{equation*}
\end{lemma}
\begin{lemma}\label{Lemma2.7}
Let $r>1$ be fixed and $0<p\le \infty$. There exist $C_1, C_2>0$ such that for $n\ge 1$ and $P\in\mathcal{P}_m$,
\begin{equation*}
\|(Pw)(x)\|_{L_p(|x|\ge a_{rm})}
\le C_2 \exp\left(-C_1\frac{m}{\sqrt{T(a_m)}}\right)\|Pw\|_{L_p(|x|\le a_{m})}.
\end{equation*}
\end{lemma}
\begin{proof}
From Lemma \ref{Lemma2.1}(4),
we can choose a constant $0 <C<1$  satisfying
\begin{equation*}
 a_m\left(1+\frac{C}{T(a_m)}\right)\le a_{rm}.
\end{equation*}
Putting $\tau:=\frac{C}{T(a_m)}$, we see
 $ a_m(1+\tau)\le a_{rm}$.
By \cite[Theorem 6.4]{[3]} there exist $C_3, C_4>0$ such that
for $m\ge 1, \tau\in (0,\frac{1}{T(a_m)}]$  and polynomial $P\in\mathcal{P}_m$,
\begin{equation}\label{eq2.2}
\|(Pw)(x)\|_{L_p(|x|\ge a_m(1+\tau))}
\le C_4 \exp\left(-C_3mT(a_m)\tau^{3/2}\right)\|Pw\|_{L_p(|x|\le a_m)}.
\end{equation}
So from (\ref{eq2.2}) we have for some $C_1>0$,
\begin{eqnarray*}
\|(Pw)(x)\|_{L_p(|x|\ge a_{rm})}&\le&\|(Pw)(x)\|_{L_p(|x|\ge a_m(1+\tau))}\\
&\le& C_4 \exp\left(-C_3C\frac{m}{\sqrt{T(a_m)}}\right)\|(Pw)(x)\|_{L_p(|x|\le a_{m})}.
\end{eqnarray*}
Then we have the result putting $C_1:=C_3C$ and $C_2:=C_4$.
\end{proof}
\begin{lemma}[{\cite[Theorem 10.3]{[4]}}]\label{Lemma2.8}
Let $P\in\mathcal{P}_n$. When $0< q\le p\le \infty$, we have for some $C>0$,
\begin{equation*}
  \|w P\|_{L_q(\mathbb{R})} \le C a_n^{\frac{1}{q}-\frac{1}{p}}\|w P\|_{L_p(\mathbb{R})},
\end{equation*}
and when $0< p\le q\le \infty$, we have for some $C>0$,
\begin{equation*}
  \|w P\|_{L_q(\mathbb{R})}
 \le
C \left(\frac{n\sqrt{T(a_n)}}{a_n}\right)^{\frac{1}{p}-\frac{1}{q}}\|wP\|_{L_p(\mathbb{R})}.
\end{equation*}
\end{lemma}
\begin{lemma}[{\cite[Theorem 1.9 infinite-finite range inequality]{[4]}}]\label{Lemma2.9}
Let $0<p\le \infty$ and $r>1$. Then there exist constants $C_1,C_2>0$
such that for some $\varepsilon>0$, and $n>0, P\in P\in \mathcal{P}_n$,
\begin{equation*}
  \|Pw\|_{L_p(a_{rn}\le |x|)}\le C_1 \exp(-C_2n^\varepsilon)\|Pw\|_{L_p(|x|\le a_n)}.
\end{equation*}
\end{lemma}
\begin{lemma}\label{Lemma2.10}
Let $p$,$q>0$ and let $r>1$. Then there exist constants $C$, $C_1>0$
such that for $P\in \mathcal{P}_{[\frac{n}{2r}]}$
\begin{equation*}
  \left\{\int_{|t|\ge a_{n/2}}|(Pw)(t)|^q dt\right\}^{1/q}
\le C_1 \exp\left(-\frac{C}{4r}\frac{n}{\sqrt{T(a_n)}}\right)
\left\{\int_{|t|\le a_{n/2}}|(Pw)(t)|^p dt\right\}^{1/p}
\end{equation*}
\end{lemma}
\begin{proof}
Let $m:=\left[\frac{n}{2r}\right]$, then we see $2rm\le n$, and if we take $n$ large enough,
then we have $\frac{n}{4r}\le m$. Therefore, using Lemma \ref{Lemma2.7} and Lemma \ref{Lemma2.8},
for $P\in\mathcal{P}_m$,
\begin{eqnarray*}
&&\|(Pw)(x)\|_{L_q(a_{n/2}\le |x|)}\le\|(Pw)(x)\|_{L_q(a_{rm}\le |x|)}\\
&\le& C_1 \exp \left(-C\frac{m}{\sqrt{T(a_{m})}}\right)\|(Pw)(x)\|_{L_q(|x|\le a_{m})}\\
&\le& C_1\exp\left(-\frac{C}{4r}\frac{n}{\sqrt{T(a_n)}}\right)
\begin{cases}
a_m^{\frac{1}{q}-\frac{1}{p}}\|w P\|_{L_p(\mathbb{R})},
& 0< q< p\le \infty, \\
\left(\frac{m\sqrt{T(a_m)}}{a_m}\right)^{\frac{1}{p}-\frac{1}{q}}\|wP\|_{L_p(\mathbb{R})},
& 0< p\le q\le \infty
\end{cases}\\
&\le& C_1 \exp\left(-\frac{C_2}{4r}\frac{n}{\sqrt{T(a_n)}}\right)
\|w P\|_{L_p(\mathbb{R})},
\end{eqnarray*}
because for any fixed $\varepsilon >0$
\begin{equation*}
\max\left\{a_m^{\frac{1}{q}-\frac{1}{p}},
\left(\frac{m\sqrt{T(a_m)}}{a_m}\right)^{\frac{1}{p}-\frac{1}{q}}\right\}
\le \exp\left(\frac{\varepsilon n}{\sqrt{T(a_n)}} \right).
\end{equation*}
Now, we may estimate $\|w P\|_{L_p(\mathbb{R})}$.
Using Lemma \ref{Lemma2.9} (infinite-finite range inequality), we have
\begin{eqnarray*}
\|w P\|_{L_p(\mathbb{R})}&\le&\|w P\|_{L_p(|x|\le a_{rm})}+\|w P\|_{L_p(a_{rm}<|x|)}\\
&\le& \|w P\|_{L_p(|x|\le a_{rm})}+C_1\|wP\|_{L_p(|x|\le a_m)}\\
&\le& C_2\|w P\|_{L_p(|x|\le a_{n/2})},
\end{eqnarray*}
because $rm\le n/2$.
\end{proof}
For convenience, we let
$[a,b]:=\{x| a\le x \le b\}$ if $a\le b$, and
$[a,b]:=\{x| b\le x \le a\}$ if $b < a $.

\begin{lemma}\label{Lemma2.11}
Let $0<\delta< 1$, $f\in\mathcal{B}_\delta$, $x,t\in \mathbb{R}$.
Then we have for some $c_1>0$
\begin{eqnarray*}
&&w^\delta\left(x+t\right)\left|f\left(x+t\right)-f\left(x\right)\right|\\
&\le&
\begin{cases}
          \exp\left(c_1xQ'\left(x\right)\right)V_\delta\left([x,x+t],f\right),
          & \textrm{if} \,\,\, \quad xt<0 \,\, \textrm{and}\,\, |t|<2|x|,\\
          V_\delta\left(\left[x,x+t\right],f\right),
          & \textrm{otherwise}.
\end{cases}
\end{eqnarray*}
\end{lemma}
\begin{proof}
Let $xt\ge 0$. Then
\begin{eqnarray}\label{eq2.3} \nonumber
w^\delta(x+t)|f(x+t)-f(x)|
&\le& w^\delta(x+t)\int_{[x,x+t]}\left|df(u)\right|\\
&\le& \int_{[x,x+t]}w^\delta(u)\left|df(u)\right|
\le V_\delta\left(\left[x,x+t\right],f\right).
\end{eqnarray}
Next, let $xt<0$ and $|t|\ge 2|x|$.
Then, for $x\le u\le x+t (t>0)$ or
$x+t\le u\le x (t<0)$ we have $w^\delta(u)\ge w^\delta(x+t)$
because of $|u|\le |x+t|$. So we have (\ref{eq2.3}).
Finally, we consider the case of $xt<0$ and $|t|<2|x|$.
Let $u\in \left[x,x+t\right] (t>0)$ or $u\in \left[x+t,x\right] (t<0)$. If $u\le|x+t|$, then we simply have
\begin{equation*}
  w^\delta(x+t)\le w^\delta(u).
\end{equation*}
So we have the result as (\ref{eq2.3}). Let $|x+t|< |u|$. We see that
\begin{equation*}
  |Q(u)-Q(x+t)|\le |t||Q'(x)|< 2|x||Q'(x)|=2xQ'(x).
\end{equation*}
Hence,
\begin{equation*}
  Q(u)-2xQ'(x)\le Q(x+t),
\end{equation*}
so
\begin{equation*}
  w^\delta(x+t)\le \exp(2\delta xQ'(x))w^\delta(u).
\end{equation*}
Therefore, as (\ref{eq2.3}) we have the result.
\end{proof}
Let
\begin{equation*}
  \chi_x(t):=
\begin{cases}
          1,& \textrm{if} \,\, t\le x,\\
          0,& \textrm{otherwise}.
\end{cases}
\end{equation*}
\begin{lemma}[{\cite[Corollary 1.2.6]{[5]}}]\label{Lemma2.12}
Let $x\in \mathbb{R}$ be a fixed number, and integer $k$ be found so that $0\le k\le n+1$ and $x\in (x_{k+1,n}, x_{kn}]$. Then there exist $P:=P_x, R:=R_x\in \mathcal{P}_{2n-1}$ such that
\begin{equation}\label{eq2.4}
  R(t)\le \chi_x(t)\le P(t), \quad  t\in\mathbb{R},
\end{equation}
and
\begin{equation}\label{eq2.5}
  \int_{-\infty}^\infty |P(t)-R(t)|w^2(t)dt\le \lambda_{k+1,n}+\lambda_{kn}.
\end{equation}
Moreover, the coefficients in the polynomials $P_x$ and $R_x$ are measurable functions  of $x$ (in fact, they are step functions, that is, when $P_x(t)$ (or $R_x(t)$)$=\sum_{i=0}^{2n-1} a_i(x)t^i$ we have $a_i(x)$ which is constant in $[x_{k+1,n}, x_{k,n}]$, so $P_x$ and $R_x$ mean the polynomials $P_{(x_{k+1,n},  x_{k,n}]}, R_{(x_{k+1,n},  x_{k,n}]}$ defined by the interval $(x_{k+1,n},  x_{k,n}]$ which contains $x$)).
\end{lemma}
\begin{lemma}[{cf. \cite[Lemma 4.1.3]{[5]}}]\label{Lemma2.13}
For $x\in \mathbb{R}$ and $n=1,2,\cdots$,
we have
\begin{equation*}
  E_{1,n}(w;\chi_x)\le C\frac{a_n}{n}w(x).
\end{equation*}
\end{lemma}
\begin{proof} Using Lemma \ref{Lemma2.12}, we estimate $E_{1,n}(w^2;\chi_x)$.
First, let $|x|\le a_{n/3}$. Let $k$ be an integer such that $x\in\left[x_{k+1,n},x_{k,n}\right]\subset \left[-a_n,a_n\right]$. Here we note $x_{1,n}<a_n$ (see (2.3)). By Lemma \ref{Lemma2.12} we get polynomials $P$ and $R$ satisfying (\ref{eq2.4}) and (\ref{eq2.5}), so that
\begin{eqnarray*}
E_{1,2n-1}(w^2;\chi_x)
&\le&\int_{-\infty}^\infty\left[P(t)-\chi_x(t)\right]w^2(t)dt+\int_{-\infty}^\infty\left[\chi_x(t)-R(t)\right]w^2(t)dt\\
&\le& \lambda_{k+1,n}+\lambda_{k,n}.
\end{eqnarray*}
Then from Lemma \ref{Lemma2.2} and Lemma \ref{Lemma2.3}, we have for $|x|\le a_{n/3}$
\begin{equation*}
\lambda_{k+1,n}+\lambda_{k,n} \le C\varphi_n(x)w^2(x)\sim \frac{a_n}{n}\sqrt{1-\frac{|x|}{a_n}}w^2(x)\le C\frac{a_n}{n}w^2(x).
\end{equation*}
Since for the Mhaskar-Saff number $a_n(w)$ with respect to the weight $w(x)=\exp(-Q(x))$ we see
$a_{n}(w^{1/2})=a_{2n}(w)$, we have
\begin{equation*}
  E_{1,2n-1}(w;\chi_x)=E_{1,2n-1}((w^{1/2})^2;\chi_x)\le C\frac{a_{n}(w^{1/2})}{n}w(x)= C\frac{a_{2n}(w)}{n}w(x).
\end{equation*}
Thus, we have the result when $|x|\le a_{n/3}$.
Next, let $|x|\ge a_{n/3}$. Then, by Lemma \ref{Lemma2.1} (2) we have
\[
Q'(x)\ge Q'(a_{n/3})\ge C \frac{n\sqrt{T(a_{n/3})}}{2a_{n/3}}\ge C \frac{n\sqrt{T(a_n)}}{a_n}.
\]
Since $Q'$ is increasing,
\begin{eqnarray*}
E_{1,2n-1}(w;\chi_x)
&\le& \int_{-\infty}^\infty(1-\chi_x(t))w(t)dt=\int_x^\infty \exp(-Q(t))dt\\
&\le& \frac{-1}{Q'(x)}\int_x^\infty (-Q'(t))\exp(-Q(t))dt=\frac{w(x)}{Q'(x)}\\
&\le& C \frac{a_n}{n\sqrt{T(a_n)}}w(x).
\end{eqnarray*}
Here, we can replace $E_{1,2n-1}(w;\chi_x)$ with $E_{1,n}(w;\chi_x)$.
\end{proof}
\begin{lemma}\label{Lemma2.14}
Let
\begin{equation*}
  \Lambda_n(t):=\int_t^\infty p_n(v)w^2(v)dv,  \quad t\in\mathbb{R}.
\end{equation*}
Let $0<\delta < 1$. Then there exist constants $0 < d \le 1$ and $C>0$ such that
\begin{equation}\label{eq2.6}
  |\Lambda_n(t)|\le C \frac{\sqrt{a_n}}{n}w^\delta(t), \quad |t| \le a_{dn}.
\end{equation}
\end{lemma}
\begin{proof} We consider the case of $n$ large enough.
We use $r > 1$ and $P \in \mathcal{P}_{\left[\frac{n}{2r}\right]}$ in Lemma \ref{Lemma2.7}.
By Lemma \ref{Lemma2.10}, we have that for $t\in \mathbb{R}$ and $n=1,2,\ldots,$
\begin{equation*}
 E_{1,n}(w^{\delta};\chi_t)\le C\frac{a_{n/\delta}}{n/\delta}w^{\delta}(t)
 \le C \frac{a_n}{n}w^{\delta}(t).
\end{equation*}
So there exists $P\in \mathcal{P}_m$, $m=\left[\frac{n}{2r}\right]$ such that
\begin{equation}\label{eq2.7}
\int_{\mathbb{R}} |\chi_t(u)-P(u)|w^\delta(u)du \le C\frac{a_n}{n}w^{\delta}(t).
\end{equation}
Here, we note that for $n$ large enough, $\frac14 n \le rm \le \frac12n$.
Hence, by the orthogonal polynomial $p_n$ and $P-1\in\mathcal{P}_{n-1}$, we have
\begin{eqnarray*}
\left|\Lambda_n(t)\right|&=&\left|\int_{-\infty}^\infty (1-\chi_t(u))p_n(u)w^2(u)du\right|\\
&=& \left|\int_{-\infty}^\infty (\chi_t(u)-P(u))p_n(u)w^2(u)du\right|\\
&\le& \int_{|u|\le a_{n/2}} \left|(\chi_t(u)-P(u))p_n(u)w^2(u)\right|du \\
&& \qquad +\int_{|u|\ge a_{n/2}} \left|(\chi_t(u)-P(u))p_n(u)w^2(u)\right|du\\
&=&:J_1+J_2.
\end{eqnarray*}
By Lemma \ref{Lemma2.5} and (\ref{eq2.7})  we see
\begin{eqnarray}\label{eq2.8}\nonumber
J_1
&\le& C \frac{1}{\sqrt{a_n}} \int_{|u|\le a_{n/2}} |(\chi_t(u)-P(u))\left|1-\frac{|u|}{a_n}\right|^{-1/4}w(u)|du\\\nonumber
&\le& C \frac{1}{\sqrt{a_n}} \int_{|u|\le a_{n/2}} |(\chi_t(u)-P(u))w^\delta(u)|du\\
&\le& C \frac{\sqrt{a_n}}{n}w^\delta(t).
\end{eqnarray}
Here we used the fact that
$ \left|1-\frac{|u|}{a_n}\right|^{-1/4}\le Cw^{\delta-1} (u)$ for $|u| \le a_{n/2}$,
and the Mhaskar-Rakhmanov-Saff number for the weight $w^\delta (x)$ is $a_{n/\delta}$.

   Next, we estimate $J_2$. From (\ref{eq2.7}), we know
\begin{equation}\label{eq2.9}
\int_{\mathbb{R}}\left|P(u)\right|w(u)du
\le \int_{\mathbb{R}}|\chi_t(u)-P(u)|w(u)du+\int_{\mathbb{R}}|\chi_t(u)|w(u)du\le C.
\end{equation}
Since $1-P, P\in\mathcal{P}_{\left[\frac{n}{2r}\right]}$,
using Lemma \ref{Lemma2.7}, Lemma \ref{Lemma2.8} with $p=1$, $q=2$ and (\ref{eq2.9}), we have
\begin{eqnarray}\label{eq2.10} \nonumber
&&\left\{\int_{|u|\ge a_{n/2}}\left|P(u)\right|^2w^2(u)du\right\}^{1/2}
\le \left\{\int_{|u|\ge a_{r\left[\frac{n}{2r}\right]}}\left|P(u)\right|^2w^2(u)du\right\}^{1/2}\\ \nonumber
&\le& C_1 \exp\left(-C\frac{n}{\sqrt{T(a_n)}}\right)
\left\{\int_{\mathbb{R}}\left|P(u)\right|^2w^2(u)du\right\}^{1/2}\\ \nonumber
&\le& C_1 \exp\left(-C\frac{n}{\sqrt{T(a_n)}}\right)\left\{\frac{n\sqrt{T(a_n)}}{a_n}\right\}^{1/2}
\int_{\mathbb{R}}|(P(v))w(v)| dv\\
&\le& C_1 \exp\left(-C_2\frac{n}{\sqrt{T(a_n)}}\right),
\end{eqnarray}
and similarly,
\begin{equation}\label{eq2.11}
  \left\{\int_{|u|\ge a_{n/2}}|1-P(u)|^2w^2(u)du\right\}^{1/2}\le C_1 \exp\left(-C_2\frac{n}{\sqrt{T(a_n)}}\right).
\end{equation}
Since we know
\begin{equation*}
|\chi_t(u)-P(u)|^2\le \left(|1-P(u)|^2+|P(u)|^2\right),
\end{equation*}
using the Schwarz inequality, we see from (\ref{eq2.10}) and (\ref{eq2.11})
\begin{eqnarray*}
J_2 &\le& \left(\int_{|u|\ge a_{n/2}}|\chi_t(u)-P(u)|^2w^2(u)du\right)^{1/2}
\left(\int_{-\infty}^\infty p_n^2(u)w^2(u)du\right)^{1/2}\\
&=& \left(\int_{|u|\ge a_{n/2}}|\chi_t(u)-P(u)|^2w^2(u)du\right)^{1/2} \\
&\le& C \left\{ \left(\int_{|u|\ge a_{n/2}}|1-P(u)|^2w^2(u)du\right)^{1/2}
+\left(\int_{|u|\ge a_{n/2}}|P(u)|^2w^2(u)du\right)^{1/2} \right\}\\
&\le& C_1\exp\left(-C_2\frac{n}{\sqrt{T(a_n)}}\right).
\end{eqnarray*}
Here we will show that there exists $0 < d <1$ such that
\begin{equation}\label{eq2.12}
  \exp\left(-\frac{C_2}{2}\frac{n}{\sqrt{T(a_n)}}\right)\le w(t), \quad  |t|\le a_{dn}.
\end{equation}
By Lemma \ref{Lemma2.1} (3), (4), we have for some constant $0< L <1$
\begin{equation*}
\frac{Q(a_{t/2})}{Q(a_{t})}\le \left(\frac{a_{t/2}}{a_{t}}\right)^{\max\left\{\Lambda,C_3T(a_t)\right\}}\le \left(1-\frac{C_4}{T(a_t)}\right)^{\max\left\{\Lambda,C_3T(a_t)\right\}} \le L < 1.
\end{equation*}
Then for a positive integer $k$,
\begin{equation*}
\frac{Q\left(a_{\frac{t}{2^k}}\right)}{Q(a_t)} \le L^k.
\end{equation*}
It means that $\frac{Q\left(a_{\frac{t}{2^k}}\right)}{Q(a_{t})}\to 0$ as $k\to \infty$.
Therefore, we see that for any constant $C>0$, there exists a constant $0< d <1$ such that
\begin{equation}\label{eq2.13}
\frac{Q(a_{dn})}{Q(a_{n})} \le C.
\end{equation}
From Lemma \ref{Lemma2.1} (2) and (\ref{eq2.13}), there exist constants $C_5>0$ and $0 < d \le 1$ such that
\[
\frac{C_2}{2}\frac{n}{\sqrt{T(a_n)}} \ge C_5 Q(a_n) \ge Q(a_{dn}).
\]
Thus, (\ref{eq2.12}) is proved.
Therefore, there exists a constant $0<d\le 1$ such that
\begin{equation}\label{eq2.14}
J_2
\le C_1\exp\left(-\frac{C_2}{2}\frac{n}{\sqrt{T(a_n)}}\right)w(t), \quad |t|\le a_{dn}.
\end{equation}
Here, by (\ref{eq2.1}) we see that for $n$ large enough,
\begin{equation}\label{eq2.15}
  \exp\left(-\frac{C_2}{2}\frac{n}{\sqrt{T(a_n)}}\right)\le C_1\frac{1}{n} \le C_1 \frac{\sqrt{a_n}}{n}.
\end{equation}
Hence (\ref{eq2.14}) and (\ref{eq2.15}) imply
\begin{equation}\label{eq2.16}
J_2\le  C_1 \frac{\sqrt{a_n}}{n}w (t), \quad |t|\le a_{dn}.
\end{equation}
Consequently, from (\ref{eq2.8}) and (\ref{eq2.16}) we have the result (\ref{eq2.6}).
\end{proof}

\setcounter{equation}{0}
\section{Proof of Theorem \ref{Theorem1.3}}

\begin{proof}[Proof of Theorem \ref{Theorem1.3}]
We will consider only for $x\ge 0$,
because for the other cases, it can be shown similarly.
Let $w=\exp(-Q)\in \mathcal{F}(C^2+)$. Let $x\in \mathbb{R}$ be fixed, then we consider  $n\in \mathbb{N}$ large enough such as
\begin{equation}\label{eq3.1}
  |x|\le a_{dn}/6,
\end{equation}
where $a_{dn}$ is defined in (\ref{eq2.6}).
Without loss of generality, we may assume that
$f(x)=0$,
so, exchange $f(t)-f(x)$ with $f(t)$. Then we may estimate
\begin{equation*}
|s_n(f,x)|=\left|\int_{-\infty}^\infty K_n(x,t)f(t)w^2(t)dt\right|=\left|\int_{-\infty}^\infty K_n(x,x+t)f(x+t)w^2(x+t)dt\right|.
\end{equation*}
For $|t|\ge\frac{a_n}{n}$ and a fixed $x$, we define
\begin{equation}\label{eq3.2}
  a_{dn}^*:=a_{dn}-x.
\end{equation}
Noting $f(x)=0$, and by (\ref{eq1.1}), we split $s_n(f,x)$ in five terms as follows:
\begin{equation*}
  s_n(f,x)=\int_{-\infty}^\infty K_n(x,x+t)f(x+t)w^2(x+t)dt=:\sum_{k=1}^5 I_k,
\end{equation*}
where, with $H(t):=K_n(x,x+t)f(x+t)w^2(x+t)$,
\begin{eqnarray*}
&&I_1:=\int_{|t|\le \frac{a_n}{n}}H(t)dt,\quad
I_2:=\int_{-a_{dn}^*-2x}^{-a_n/n}H(t)dt, \quad I_3:=\int_{a_n/n}^{a_{dn}^*}H(t)dt,\\
&&I_4:=\int_{-\infty}^{-a_{dn}^*-2x}H(t)dt, \quad  I_5:=\int_{a_{dn}^*}^\infty H(t)dt.
\end{eqnarray*}
First, we estimate $I_1$. Using the Schwarz inequality and the estimates on the Christoffel functions from Lemma \ref{Lemma2.2} and Lemma \ref{Lemma2.4}(note $\varphi_n(x) \sim n/a_n$ under the assumption (\ref{eq3.1})),
\begin{eqnarray*}
K_n^2(x,x+t)&\le& K_n(x,x)K_n(x+t,x+t)\\
&\le& C\varphi_n^{-1}(x)\varphi_n^{-1}(x+t)w^{-2}(x)w^{-2}(x+t)\\
&\le& C \left(\frac{n}{a_n}\right)^2\sqrt{T(x+t)}w^{-2}(x)w^{-2}(x+t).
\end{eqnarray*}
Therefore, we have
\begin{equation}\label{eq3.3}
  H(t)\le C \frac{n}{a_n}w^{-1}(x)f(x+t)T^{1/4}(x+t)w(x+t).
\end{equation}
Hence, for some $0 < \delta < 1$ we have
\begin{eqnarray}\nonumber
|I_1|&\le&  C \frac{n}{a_n}w^{-1}(x)
\int_{|t|\le a_n/n}\left|f(x+t)w^\delta(x+t)\right|dt\\\nonumber
&\le& C\frac{n}{a_n}w^{-1}(x)
\int_{|t|\le a_n/n}w^\delta(x+t)\int_{[x,x+t]}\left|df(u)\right|dt
\end{eqnarray}
(note $f(x)=0$).
By Lemma \ref{Lemma2.3} (b) we have $w(x+t)\sim w(x)$ (note (\ref{eq3.1})), so
\begin{eqnarray}\label{eq3.4}\nonumber
|I_1|&\le&  C \frac{n}{a_n}w^{-1}(x)
\int_{|t|\le a_n/n}\int_{[x,x+t]}w^\delta(u)\left|df(u)\right|dt\\\nonumber
&\le& C \frac{n}{a_n}w^{-1}(x)\int_{|t|\le a_n/n}V_\delta\left(\left[x,x+t\right],f\right)dt\\
&\le& Cw^{-1}(x)V_\delta\left(\left[x-\frac{a_n}{n},x+\frac{a_n}{n}\right],f\right).
\end{eqnarray}
Secondly, we estimate $I_3$.
By (\ref{eq1.2}) we have
\begin{equation*}
  K_n(x,x+t)=\frac{\gamma_{n-1}}{\gamma_n}\frac{p_{n-1}(x)p_n(x+t)-p_n(x)p_{n-1}(x+t)}{t}.
\end{equation*}
Using this, we estimate $I_3$. We see that
\begin{equation*}
  I_3:=\frac{\gamma_{n-1}}{\gamma_n}\left\{p_{n-1}(x)I_{3,1}-p_n(x)I_{3,2}\right\},
\end{equation*}
where
\begin{eqnarray*}
&&I_{3,1}:=\int_{a_n/n}^{a_{dn}^*} p_n(x+t)\frac{f(x+t)}{t}w^2(x+t)dt,\\
&&I_{3,2}:=\int_{a_n/n}^{a_{dn}^*} p_{n-1}(x+t)\frac{f(x+t)}{t}w^2(x+t)dt.
\end{eqnarray*}
From $\gamma_{n-1}/\gamma_n\sim a_n$ (see Lemma \ref{Lemma2.6}), we have
\begin{equation*}
  |I_3|\le C a_n^{1/2}w^{-1}(x)\left\{|I_{3,1}|+|I_{3,2}|\right\},
\end{equation*}
because we see, from Lemma \ref{Lemma2.5}, that for $|x| \le a_{dn}/6$,
\[
\max \left\{\left|p_n(x)\right|, \left|p_{n-1}(x)\right|\right\}
\le Ca_n^{-1/2}w^{-1}(x)\left|1-\frac{|x|}{a_n}\right|^{-1/4}
\le C a_n^{-1/2}w^{-1}(x).
\]

We use $\Lambda_n(x)$ in Lemma \ref{Lemma2.14}. Applying integration by parts, we have
\begin{eqnarray*}
&&I_{3,1}=\frac{n}{a_n}f\left(x+\frac{a_n}{n}\right)\Lambda_n\left(x+\frac{a_n}{n}\right)
-\frac{1}{a_{dn}^*}f(x+a_{dn}^*)\Lambda_n(x+a_{dn}^*)\\
&&   -\int_{a_n/n}^{a_{dn}^*} \frac{\Lambda_n(x+t)|df(x+t)|}{t}+\int_{a_n/n}^{a_{dn}^*} \frac{\Lambda_n(x+t)f(x+t)}{t^2}dt.
\end{eqnarray*}
When $0< t\le a_{dn}^*$, we see that $|x+t|\le a_{dn}$ (see (\ref{eq3.2})). Hence, by (\ref{eq2.6})
we have
\begin{eqnarray}\label{eq3.5}\nonumber
\sqrt{a_n}|I_{3,1}|
&\le&  C\bigg\{ \left|f\left(x+\frac{a_n}{n}\right)\right|w^\delta\left(x+\frac{a_n}{n}\right)
+\frac{1}{n}|f(a_{dn})|w^\delta(a_{dn})\\
&&   +\frac{a_n}{n}\int_{a_n/n}^{a_{dn}^*} \frac{w^\delta(x+t)|df(x+t)|}{t} +\frac{a_n}{n}\int_{a_n/n}^{a_{dn}^*}\frac{w^\delta(x+t)|f(x+t)|}{t^2}dt\bigg\}.
\end{eqnarray}
Since $V_\delta(\left[x,x+\frac{a_n}{n}\right]) \le V_\delta(\left[x,x+\frac{a_n}{k}\right])$ for $1 \le k \le n$,
we have from Lemma \ref{Lemma2.11}
\begin{eqnarray}\label{eq3.6}\nonumber
w^\delta\left(x+\frac{a_n}{n}\right)\left|f\left(x+\frac{a_n}{n}\right)\right|
&\le& C V_\delta \left(\left[x,x+\frac{a_n}{n}\right],f\right)\\
&\le& C \frac{1}{n}\sum_{k=1}^n V_\delta\left(\left[x,x+\frac{a_n}{k}\right],f\right)
\end{eqnarray}
and
\begin{equation}\label{eq3.7}
|f(a_{dn})|w^\delta(a_{dn})\le C V_\delta(\left[x,a_{dn}\right],f)
\end{equation}
(note $f(x)=0$).
On the other hand, substituting $u=\frac{a_n}{t}$ and
decreasing of $V_\delta(\left[x,x+\frac{a_n}{u}\right],f)$ for $u$,
we have
\begin{eqnarray}\label{eq3.8} \nonumber
&&\int_{a_n/n}^{a_{dn}^*} \frac{w^\delta(x+t)|f(x+t)|}{t^2}dt
\le C \int_{a_n/n}^{a_{dn}^*} \frac{V_\delta\left(\left[x,x+t\right],f\right)}{t^2}dt\\ \nonumber
&& \quad =\frac{1}{a_n}\int_{a_n/a_{dn}^*}^n V_\delta\left(\left[x,x+\frac{a_n}{u}\right],f\right)du
\le \frac{1}{a_n}\sum_{k=1}^n\int_{k}^{k+1} V_\delta\left(\left[x,x+\frac{a_n}{u}\right],f\right)du \\
&& \quad \le \frac{1}{a_n}\sum_{k=1}^n V_\delta\left(\left[x,x+\frac{a_n}{k}\right],f\right).
\end{eqnarray}
We estimate the remaining term in (\ref{eq3.5}). Using integration by parts  and (\ref{eq3.8}), we have
\begin{eqnarray}\label{eq3.9}
&&\int_{a_n/n}^{a_{dn}^*} \frac{w^\delta(x+t)|df(x+t)|}{t} \\\nonumber
&=&
\frac{1}{a_{dn}^*}V_\delta\left(\left[x,x+a_{dn}^*\right],f\right)
-\frac{n}{a_n}V_\delta\left(\left[x,x+\frac{a_n}{n}\right],f\right)
+\int_{a_n/n}^{a_{dn}^*} \frac{V_\delta\left(\left[x,x+t\right],f\right)}{t^2}dt\\ \nonumber
&\le& C \frac{1}{a_n}\sum_{k=1}^n V_\delta\left(\left[x,x+\frac{a_n}{k}\right],f\right).
\end{eqnarray}
Hence, substituting (\ref{eq3.6}), (\ref{eq3.7}), (\ref{eq3.8}) and (\ref{eq3.9}) into (\ref{eq3.5}), we get
\begin{equation}\label{eq3.10}
a_n^{1/2}|I_{3,1}|
\le C\frac{1}{n}\left(\sum_{k=1}^n V_\delta\left(\left[x,x+\frac{a_n}{k}\right],f\right)
+ V_\delta(\left[x,a_{dn}\right],f)\right).
\end{equation}
For $I_{3,2}$ we obtain the estimate as (\ref{eq3.10}), so we have for a constant $\alpha>0$,
\begin{equation}\label{eq3.11}
|I_{3}|\le C w^{-1}(x)
\frac{1}{n}\left(\sum_{k=1}^n V_\delta\left(\left[x,x+\frac{a_n}{k}\right],f\right)
+ V_\delta(\left[x,a_{dn}\right],f)\right).
\end{equation}
Thirdly, we estimate $I_5$. Using (\ref{eq3.3}), we have
\begin{eqnarray}\label{eq3.12}\nonumber
|I_5|&\le&  C \frac{n}{a_n}w^{-1}(x)\int_{a_{dn}^*}^\infty w(x+t)T^{1/4}(x+t)|f(x+t)|dt\\
&\le& C\frac{n}{a_n}w^{-1}(x)\left\{I_{5,1}+I_{5,2}\right\},
\end{eqnarray}
where
\begin{eqnarray*}
&&I_{5,1}:=\int_{a_{dn}^*}^\infty w(x+t)T^{1/4}(x+t)\int_0^{a_{dn}^*}|df(x+u)|dt,\\
&&I_{5,2}:=\int_{a_{dn}^*}^\infty w(x+t)T^{1/4}(x+t)\int_{a_{dn}^*}^t|df(x+u)|dt
\end{eqnarray*}
(note $f(x+t)=\int_0^tdf(x+u)$). Since $a_{dn}-x=a_{dn}^*\le t$,
we see for $x+t=a_s \ge a_{dn}$
\begin{equation}\label{no_A}
\frac{T^{1/4}(a_s)}{Q'(a_s)} \sim \frac{a_s}{sT^{1/4}(a_s)} \le C\frac{a_n}{nT^{1/4}(a_n)},
\end{equation}
because
\[
\frac{a_s/s}{a_n/n} \le C \frac{n}{s} \left(\frac{s}{n}\right)^{1/\Lambda}=\left(\frac{n}{s}\right)^{1-1/\Lambda} \le C,
\]
(see Lemma \ref{Lemma2.1} (6)). Then for every $u\ge a_{dn}^*$,
\begin{eqnarray}\label{eq3.13}\nonumber
\int_u^\infty w(x+t)T^{1/4}(x+t)dt
&=&  \int_u^\infty \frac{T^{1/4}(x+t)}{Q'(x+t)}Q'(x+t)w(x+t)dt\\
&\le&C\frac{a_n}{nT^{1/4}(a_n)}w(x+u).
\end{eqnarray}
Therefore, with integration by parts and (\ref{eq3.13}),
\begin{eqnarray}\label{eq3.14}\nonumber
|I_{5,2}|
&=& \left|\int_{a_{dn}^*}^\infty \int_u^\infty w(x+t)T^{1/4}(x+t)dt |df(x+u)| \right| \\ \nonumber
&\le& C \frac{a_n}{nT^{1/4}(a_n)}\int_{a_{dn}^*}^\infty w(x+u)|df(x+u)|\\\nonumber
&\le& C \frac{a_n}{nT^{1/4}(a_n)}V_1\left(\left[x+a_{dn}^*,\infty\right],f\right)\\
&=&C\frac{a_n}{nT^{1/4}(a_n)}V_1\left(\left[a_{dn},\infty\right],f\right),
\end{eqnarray}
and using (\ref{eq3.13}) with $u=a_{dn}^*$,
\begin{equation}\label{eq3.15}
|I_{5,1}|\le C \frac{a_n}{nT^{1/4}(a_n)}w(x+a_{dn}^*)\int_0^{a_{dn}^*}|df(x+u)|.
\end{equation}
Let
\[
x+a_{\frac{dn}{2}}^*=a_{\frac{dn}{2}}.
\]
Therefore,
there exists $\varepsilon>0$ such that
\begin{eqnarray}\label{eq3.16}\nonumber
w(x+a_{dn}^*)\int_0^{a_{\frac{dn}{2}}^*}|df(x+u)|
&\le& \frac{w(x+a_{dn}^*)}{w(x+a_{\frac{dn}{2}}^*)}
\int_0^{a_{\frac{dn}{2}}^*}w(x+u)|df(x+u)|\\\nonumber
&\le& \frac{w(a_{dn})}{w(a_{\frac{dn}{2}})}V_1\left(\left[x,a_{\frac{dn}{2}}\right],f\right)\\
&\le& C \exp(-cn^{\varepsilon})V_1\left(\left[x,a_{\frac{dn}{2}}\right],f\right).
\end{eqnarray}
The last inequality holds as follows.
\begin{equation*}
Q(a_{dn})-Q\left(a_{\frac{dn}{2}}\right)\ge Q'(a_{\frac{dn}{2}})\left(a_{dn}-a_{\frac{dn}{2}}\right)
\ge CnT^{-1/2}(a_n) \ge C n^{\varepsilon}, \quad \textrm{for some } \varepsilon >0.
\end{equation*}
Therefore we have the last inequality in (\ref{eq3.16}).
Since we consider only $n$ such that $|x|\le a_{dn}/6$ when $a_{\frac{dn}{2}}^*\le u\le a_{dn}^*$, 
we have $0\le x+u\le x+a_{dn}^*=a_{dn}$. Moreover, there exists $c_1>0$ such that
\begin{eqnarray}\label{eq3.17}\nonumber
&&w(x+a_{dn}^*)\int_{a_{\frac{dn}{2}}^*}^{a_{dn}^*}|df(x+u)|\le \int_{a_{\frac{dn}{2}}^*}^{a_{dn}^*}w(x+u)|df(x+u)|\\
&\le& C V_1\left(\left[(x+a_{\frac{dn}{2}}^*,\infty\right],f\right)
= C V_1\left(\left[a_{\frac{dn}{2}},\infty\right],f\right).
\end{eqnarray}
Substituting (\ref{eq3.16}), (\ref{eq3.17}) into (\ref{eq3.15}), we have
\begin{equation}\label{eq3.18}
|I_{5,1}|\le C \frac{a_n}{nT^{1/4}(a_n)}
\left(\exp(-cn^{\varepsilon/2})V_1\left(\left[x,a_{\frac{dn}{2}}\right],f\right)
+V_1\left(\left[a_{\frac{dn}{2}},\infty\right),f\right)\right).
\end{equation}
Together with (\ref{eq3.14}), (\ref{eq3.18}) and (\ref{eq3.12}) we have for a constant $c_1>0$,
\begin{eqnarray}\label{eq3.19}\nonumber
|I_{5}|
&\le&
CT^{-1/4}(a_n)w^{-1}(x)\\ \nonumber
&&\times\left( \exp(-cn^{\varepsilon/2})V_1\left(\left[x,x+a_{\frac{dn}{2}}\right],f\right)
+V_1\left(\left[a_{\frac{dn}{2}},\infty\right),f\right)
\right)\\
&\le& CT^{-1/4}(a_n)w^{-1}(x)
\left(\frac{1}{n}V_1\left(\left[x,a_{\frac{dn}{2}}\right],f\right)
+V_1\left(\left[a_{\frac{dn}{2}},\infty\right),f\right)\right).
\end{eqnarray}
Fourth, we can obtain an estimate of $I_2$ as $I_3$. But we need to notice slightly.
Let us define
\begin{equation*}
  I_2:=\frac{\gamma_{n-1}}{\gamma_n}\left\{p_{n-1}(x)I_{2,1}-p_n(x)I_{2,2}\right\},
\end{equation*}
where
\begin{eqnarray*}
&&I_{2,1}:=\int_{-a_{dn}^*-2x}^{-a_n/n} p_n(x+t)\frac{f(x+t)}{t}w^2(x+t)dt,\\
&&I_{2,2}:=\int_{-a_{dn}^*-2x}^{-a_n/n} p_{n-1}(x+t)\frac{f(x+t)}{t}w^2(x+t)dt.
\end{eqnarray*}
Then we have
\begin{equation*}
  |I_2|\le C a_n^{1/2}w^{-1}(x)\left\{|I_{2,1}|+|I_{2,2}|\right\}.
\end{equation*}
The formula corresponding to (\ref{eq3.5}) is
\begin{eqnarray}\label{eq3.20}
&&\sqrt{a_n}|I_{2,1}|
\le  C\bigg\{\left|f\left(x-\frac{a_n}{n}\right)\right|w^\delta\left(x-\frac{a_n}{n}\right)
+\frac{1}{n}|f(-a_{dn})|w^\delta(-a_{dn})\\\nonumber
&&   +\frac{a_n}{n}\int_{-a_{dn}^*-2x}^{-a_n/n} \frac{w^\delta(x+t)|df(x+t)|}{t}
+\frac{a_n}{n}\int_{-a_{dn}^*-2x}^{-a_n/n}\frac{w^\delta(x+t)|f(x+t)|}{t^2}dt\bigg\}.
\end{eqnarray}
As (\ref{eq3.6}) we have, using Lemma \ref{Lemma2.11},
\begin{equation}\label{eq3.21}
  w^\delta\left(x-\frac{a_n}{n}\right)\left|f\left(x-\frac{a_n}{n}\right)\right|
  \le C\exp(c_1xQ'(x)) \frac{1}{n}\sum_{k=1}^n V_\delta\left(\left[x-\frac{a_n}{k},x\right],f\right).
\end{equation}
Since $a_n\ge a_{dn}\ge 2x>0$, using Lemma \ref{Lemma2.11} with $f(-a_{dn})-f(x)=f(-a_{dn})$,
\begin{equation}\label{eq3.22}
|f(-a_{dn})|w^\delta(-a_{dn})
\le CV_\delta(\left[-a_{dn},x\right],f).
\end{equation}
From Lemma \ref{Lemma2.11} again
\begin{equation*}
\int_{-a_{dn}^*-2x}^{-a_n/n} \frac{w^\delta(x+t)|f(x+t)|}{t^2}dt\le C \exp(c_1xQ'(x))\int_{-a_{dn}^*-2x}^{-a_n/n} \frac{V_\delta(\left[x+t,x\right],f)}{t^2}dt.
\end{equation*}
Here we note that in the case of $|t|\le 2x, xt<0$ we need the factor $\exp(c_1xQ'(x))$. Let $u:=-\frac{a_n}{t}$. Then noting the fact that $V_\delta(\left[x-\frac{a_n}{u},x\right],f)$ is a decreasing function of $u$, we have
\begin{eqnarray}\label{eq3.23}
\int_{-a_{dn}^*-x}^{-a_n/n} \frac{V_\delta(\left[x+t,x\right],f)}{t^2}dt
&=&\frac{1}{a_n}\int_{\frac{a_n}{a_{dn}}}^{n} V_\delta(\left[x-\frac{a_n}{u},x\right],f)du\\ \nonumber
&\le&
\frac{1}{a_n}\sum_{k=1}^n V_\delta\left(\left[x-\frac{a_n}{k},x\right],f\right)
\end{eqnarray}
and
\begin{eqnarray*}
\int_{-a_{dn}^*-2x}^{-a_{dn}^*-x} \frac{V_\delta(\left[x+t,x\right],f)}{t^2}dt
&\le&\frac{1}{a_{dn}}V_\delta(\left[-a_{dn},x\right],f).
\end{eqnarray*}
Therefore, with (\ref{eq3.23}) we have
\begin{eqnarray}\label{eq3.24} \nonumber
&&\int_{-a_{dn}^*}^{-a_n/n} \frac{w^\delta(x+t)|f(x+t)|}{t^2}dt \\
&\le& C\frac{\exp(c_1xQ'(x))}{a_n}
\left(\sum_{k=1}^n V_\delta\left(\left[x-\frac{a_n}{k},x\right],f\right)
+V_\delta(\left[-a_{dn},x\right],f) \right).
\end{eqnarray}
We estimate the remaining term in (\ref{eq3.20}). Using integration by parts  and (\ref{eq3.24}), we have
\begin{eqnarray}\label{eq3.25}\nonumber
&&\int_{-a_{dn}^*-2x}^{-a_n/n} \frac{w^\delta(x+t)|df(x+t)|}{t}\\ \nonumber
&=&\frac{1}{a_{dn}^*+2x}V_\delta\left(\left[-a_{dn},x\right],f\right)-\frac{n}{a_n}V_\delta \left(\left[x-\frac{a_n}{n},x\right],f\right)\\\nonumber
&& +\int_{-a_{dn}^*-2x}^{-a_n/n} \frac{w^\delta(x+t)|f(x+t)|}{t^2}dt\\
&\le&
C_1\left(\frac{1}{a_{dn}}V_\delta\left(\left[-a_{dn},x\right],f\right)
+\frac{n}{a_n}V_\delta \left(\left[x-\frac{a_n}{n},x\right],f\right)\right) \\ \nonumber
&& +C_2\frac{\exp(c_1xQ'(x))}{a_n}
\left(\sum_{k=1}^n V_\delta\left(\left[x-\frac{a_n}{k},x\right],f\right)
+V_\delta(\left[-a_{dn},x\right],f) \right).
\end{eqnarray}
Hence, substituting (\ref{eq3.21}) (\ref{eq3.22}), (\ref{eq3.24}) and (\ref{eq3.25}) into (\ref{eq3.20}), we get
\begin{equation}\label{eq3.26}
\sqrt{a_n}|I_{2,1}|
\le C\frac{\exp(c_1xQ'(x))}{n}
\left(\sum_{k=1}^n V_\delta\left(\left[x-\frac{a_n}{k},x\right],f\right)
+V_\delta(\left[-a_{dn},x\right],f) \right).
\end{equation}
Similarly, for $I_{2,2}$ we obtain the estimate as (\ref{eq3.26}), so we have
\begin{eqnarray}\label{eq3.27} \nonumber
|I_{2}|&\le& C w^{-1}(x)\frac{\exp(c_1xQ'(x))}{n}\\
&& \times \left(\sum_{k=1}^n V_\delta\left(\left[x-\frac{a_n}{k},x\right],f\right)
+V_\delta(\left[-a_{dn},x\right],f) \right).
\end{eqnarray}
Lastly, the estimate of $I_4$ also is obtained as $I_5$. Using (\ref{eq3.3}), we have
\begin{eqnarray}\label{eq3.28}\nonumber
|I_4| &\le& C \frac{n}{a_n}w^{-1}(x)\int_{-\infty}^{-a_{dn}^*-2x} w(x+t)|(T^{1/4}f)(x+t)|dt\\
&\le& C\frac{n}{a_n}w^{-1}(x)\left\{I_{4,1}+I_{4,2}\right\},
\end{eqnarray}
where
\begin{eqnarray*}
&&I_{4,1}:=\int_{-\infty}^{-a_{dn}^*-2x} w(x+t)T^{1/4}(x+t)\int_{-a_{dn}^*-2x}^0|df(x+u)|dt,\\
&&I_{4,2}:=\int_{-\infty}^{-a_{dn}^*-2x} w(x+t)T^{1/4}(x+t)\int_t^{-a_{dn}^*-2x}|df(x+u)|dt
\end{eqnarray*}
(note $f(x+t)=\int_0^tdf(x+u)$). Let $t\le -a_{dn}^*-2x$.
For every $u\le -a_{dn}$,
\begin{eqnarray}\label{eq3.29}\nonumber
\int_{-\infty}^u w(x+t)T^{1/4}(x+t)dt
&=& \left|\int_{-\infty}^u w(x+t)Q'(x+t)\frac{T^{1/4}(x+t)}{Q'(x+t)}dt\right|\\
&\le& C\frac{a_n}{nT^{1/4}(a_n)}w(x+u)
\end{eqnarray}
because for $x+t=-a_s \le -a_{dn}$
\[
\frac{T^{1/4}(a_s)}{Q'(a_s)}  \le C\frac{a_n}{nT^{1/4}(a_n)},
\]
(see (\ref{no_A})). Therefore, with integration by parts and (\ref{eq3.29}),
\begin{eqnarray}\label{eq3.30}\nonumber
|I_{4,2}|
&\le& C\frac{a_n}{nT^{1/4}(a_n)}\int_{-\infty}^{-a_{dn}^*-2x} w(x+u)|df(x+u)|\\
&\le& C\frac{a_n}{nT^{1/4}(a_n)}V_1(\left[-\infty,-a_{dn}\right],f).
\end{eqnarray}
Just like (\ref{eq3.15}) and  (\ref{eq3.16}), 
using (\ref{eq3.29}) with $u=-a_{dn}^*$,
\begin{eqnarray}\label{eq3.31}
|I_{4,1}|
&\le& C\frac{a_n}{nT^{1/4}(a_n)}w(-a_{dn})\int_{-a_{dn}^*-2x}^0|df(x+u)|\\\nonumber
&\le& C \exp(-cn^{\varepsilon})V_1(\left[-a_{dn},x\right],f).
\end{eqnarray}
Hence, from (\ref{eq3.30}), (\ref{eq3.31}) and (\ref{eq3.28}),
\begin{eqnarray}\label{eq3.32} \nonumber
&&|I_4|
\le CT^{-1/4}(a_n)w^{-1}(x)\\ \nonumber
&& \qquad \times\left(V_1(\left[-\infty,-a_{dn}\right],f)+ T^{1/2}(a_n)\exp(-cn^{\varepsilon})V_1\left(\left[-a_{\frac{dn}{2}},x\right],f\right)\right)\\
&& \qquad \le CT^{-1/4}(a_n)w^{-1}(x)
\left(V_1(\left(-\infty,-a_{dn}\right],f)+ \frac{1}{n}V_1\left(\left[-a_{\frac{dn}{2}},x\right],f\right)\right).
\end{eqnarray}
We summarize the above results. First, we note
\begin{equation*}
w^{-1}(x)=\exp\left(\frac{xQ'(x)}{T(x)}\right)\leqslant \exp\left(\frac{1}{\Lambda} xQ'(x)\right).
\end{equation*}
Hence, there exists $c>0$ such that
\begin{equation*}
w^{-1}(x)\exp(c_1xQ'(x))\leqslant \exp\left(cxQ'(x)\right).
\end{equation*}
From (\ref{eq3.4})
\begin{equation*}
|I_1|\leqslant C\exp\left(cxQ'(x)\right)\frac{1}{n}
\sum_{k=1}^nV_\delta\left(\left[x-\frac{a_n}{k},x+\frac{a_n}{k}\right],f\right).
\end{equation*}
From (\ref{eq3.11})
\begin{equation*}
|I_{3}|\le C \exp\left(cxQ'(x)\right)
\frac{1}{n}\left(\sum_{k=1}^n V_\delta\left(\left[x,x+\frac{a_n}{k}\right],f\right)
+ V_\delta(\left[x,a_{dn}\right],f)\right).
\end{equation*}
From (\ref{eq3.19})
\begin{equation*}
|I_{5}|
\le C \exp\left(cxQ'(x)\right)
\left(\frac{1}{nT^{1/4}(a_n)}V_1\left(\left[x,a_{\frac{dn}{2}}\right],f\right)
+\frac{1}{T^{1/4}(a_n)}V_1\left(\left[a_{\frac{dn}{2}},\infty\right),f\right)\right).
\end{equation*}
From (\ref{eq3.27})
\begin{equation*}
|I_{2}| \le C \exp\left(cxQ'(x)\right)\frac{1}{n}
\left(\sum_{k=1}^n V_\delta\left(\left[x-\frac{a_n}{k},x\right],f\right)
+V_\delta(\left[-a_{dn},x\right],f) \right).
\end{equation*}
From (\ref{eq3.32})
\begin{eqnarray*}
|I_4|
\le C\exp\left(cxQ'(x)\right)
\left(\frac{1}{nT^{1/4}(a_n)}V_1\left(\left[-a_{\frac{dn}{2}},x\right],f\right)
+\frac{1}{T^{1/4}(a_n)}V_1(\left(-\infty,-a_{dn}\right],f)\right).
\end{eqnarray*}
Hence, we have
\begin{equation*}
|I_1|+|I_2|+|I_3|\leqslant C\exp\left(cxQ'(x)\right)\left(\frac{1}{n}\sum_{k=1}^nV_\delta\left(\left[x-\frac{a_n}{k},x+\frac{a_n}{k}\right],f\right)
+\frac{1}{n}\int_{|u|\leqslant a_{dn}}w^\delta(u)|df(u)|\right),
\end{equation*}
and
\begin{equation*}
|I_4|+|I_5|\leqslant C\exp\left(cxQ'(x)\right)
\left(\frac{1}{nT^{1/4}(a_n)}\int_{|u|\leqslant a_{\frac{dn}{2}}}w(u)|df(u)|
+\frac{1}{T^{1/4}(a_n)}\int_{|u|\geqslant a_{\frac{dn}{2}}}w(u)|df(u)|\right).
\end{equation*}
Thus, we obtain (\ref{eq1.3}). We need to show (\ref{eq1.4}).
Let $m:=\left[\sqrt{a_n n}\right]$, then we see
\begin{eqnarray*}
&&\frac{1}{n}\sum_{k=1}^n V_\delta\left(\left[x-\frac{a_n}{k},x+\frac{a_n}{k}\right],f\right)\\
&=&\frac{1}{n}\sum_{k=1}^m V_\delta\left(\left[x-\frac{a_n}{k},x+\frac{a_n}{k}\right],f\right)
+\frac{1}{n}\sum_{k=m}^n V_\delta\left(\left[x-\frac{a_n}{k},x+\frac{a_n}{k}\right],f\right)\\
&\le&\frac{m}{n}V_\delta\left(\left[x-a_n,x+a_n\right],f\right)
+\frac{1}{n}\sum_{k=m}^n V_\delta\left(\left[x-\frac{a_n}{m},x+\frac{a_n}{m}\right],f\right)\\
&\le&\sqrt{\frac{a_n}{n}} V_\delta\left(\left[x-a_n,x+a_n\right],f\right)
+V_\delta\left(\left[x-\sqrt{\frac{a_n}{n}},x+\sqrt{\frac{a_n}{n}}\right],f\right).
\end{eqnarray*}
Therefore, the proof of (\ref{eq1.4}) is complete.
Here, it is clear that
\begin{equation*}
  \lim_{n\rightarrow \infty}\sqrt{\frac{a_n}{n}} V_\delta(\left[x-a_n,x+a_n\right],f)\le \lim_{n\rightarrow \infty}\sqrt{\frac{a_n}{n}} V_\delta(\mathbb{R},f)=0.
\end{equation*}
We have
\begin{equation*}
  V_\delta(\left[x-\varepsilon,x+\varepsilon\right],f)\le C \int_{x-\varepsilon}^{x+\varepsilon}|df(t)|,
\end{equation*}
and so
\begin{equation*}
  \lim_{\varepsilon\rightarrow 0}V_\delta(\left[x-\varepsilon,x+\varepsilon\right],f)=0.
\end{equation*}
Furthermore, we have
\begin{equation*}
\frac{1}{n}\int_{|u|\leqslant a_{dn}}w^\delta(u)|df(u)|\leqslant \frac{1}{n}V_\delta(\mathbb{R},f)\rightarrow 0
\quad \textrm{as} \quad n\rightarrow \infty,
\end{equation*}
and similarly
\begin{equation*}
\frac{1}{nT^{1/4}(a_n)}\int_{|u|\leqslant a_{\frac{dn}{2}}}w(u)|df(u)|,
\quad \frac{1}{T^{1/4}(a_n)}\int_{|u|\geqslant a_{\frac{dn}{2}}}w(u)|df(u)|\rightarrow 0,
\end{equation*}
as $n \to \infty$.
Consequently, it is proved that the sequence $\left\{s_n(f,x)\right\}$ converges to $f(x)$.
\end{proof}



\end{document}